\documentclass[11pt,reqno]{amsart}

\usepackage{amsfonts, amsthm, amsmath, amssymb}
\allowdisplaybreaks[4]

\usepackage{tikz} 

\usepackage{float} 
\usepackage{multirow} 
\usepackage{supertabular} 

\makeatletter

\newcommand{\Rmnum}[1]{\expandafter\@slowromancap\romannumeral #1@}
\makeatother

\usepackage{enumerate} 
\usepackage[inline]{enumitem} 

\usepackage[pagebackref]{hyperref} 
\hypersetup{colorlinks=true} 

\usepackage{mathabx}

\numberwithin{equation}{section}
\theoremstyle{plain}

\newtheorem{theorem}{Theorem}[section]
\newtheorem{lemma}[theorem]{Lemma}
\newtheorem{corollary}[theorem]{Corollary}

\newtheorem{proposition}[theorem]{Proposition}

\theoremstyle{definition}
\newtheorem{Def}[theorem]{Definition}

\newtheorem{remark}[theorem]{Remark}
\newtheorem{?}[theorem]{Problem}

\renewcommand{\S}{\mathfrak{S}}

\def\des{\mathrm{des}}
\def\std{\mathrm{std}}
\def\st{\mathrm{st}}
\def\inv{\mathrm{inv}}
\def\blk{\mathrm{blk}}
\def\sing{\mathrm{sing}}
\def\Par{\mathrm{Par}}

\begin{document}

\title[Consecutive and quasi-consecutive patterns]{Consecutive and quasi-consecutive patterns: $\mathrm{des}$-Wilf classifications and generating functions}

\author[Y. Wang, Q.Fang, S. Fu, S. Kitaev, H. Li]{Yan Wang, Qi Fang, Shishuo Fu, Sergey Kitaev, Haijun Li}
\address[Yan Wang]{College of Mathematics and Statistics, Chongqing University, Chongqing 401331, P.R. China}
\email{yanw2143@foxmail.com}
\address[Qi Fang]{College of Mathematics and Statistics, Chongqing University, Chongqing 401331, P.R. China}
\email{qifangpapers@163.com}
\address[Shishuo Fu]{College of Mathematics and Statistics, Chongqing University, Chongqing 401331, P.R. China}
\email{fsshuo@cqu.edu.cn}
\address[Sergey Kitaev]{Department of Mathematics and Statistics, University of Strathclyde, 26 Richmond Street, Glasgow G1 1XH, UK}
\email{sergey.kitaev@strath.ac.uk}
\address[Haijun Li]{College of Mathematics and Statistics, Chongqing University, Chongqing 401331, P.R. China}
\email{lihaijun@cqu.edu.cn}

\date{\today}

\begin{abstract}
Motivated by a correlation between the distribution of descents over permutations that avoid a consecutive pattern and those avoiding the respective quasi-consecutive pattern, as established in this paper, we obtain a complete $\des$-Wilf classification for quasi-consecutive patterns of length up to 4. For equivalence classes containing more than one pattern, we construct various descent-preserving bijections to establish the equivalences, which lead to the provision of proper versions of two incomplete bijective arguments previously published in the literature. Additionally, for two singleton classes, we derive explicit bivariate generating functions using the generalized run theorem.
\end{abstract}

\maketitle

\section{Introduction}

Permutation patterns are a very popular research topic. The introduction of this area is traditionally attributed to 
Donald Knuth, particularly to exercises on pages 242--243 of the first volume of ``The Art of Computer Programming''~\cite{knu75}. However, the first
systematic study of pattern avoidance was not conducted until the seminal paper by Simion and Schmidt~\cite{SS85}; see \cite{kit11} and the references therein for more information on the topic. 

Of interest to us is finding the distributions of permutation statistics over pattern-avoiding classes of permutations, which have attracted much attention in the literature; for example, see \cite{HK} for a recent paper in this direction, along with references therein. More specifically, we are interested in the widely studied distributions of descents in pattern-avoiding permutations, thereby extending the studies in \cite{BDGZ19, Bukata2019, HK, MR06} to the case of avoidance of consecutive and quasi-consecutive patterns (see below for definitions). The study of classical statistics such as descents and excedances over the symmetric group $\S_n$ dates back to MacMahon~\cite{mac60}, and their pivotal roles in the combinatorial interpretations of Eulerian polynomials were revealed by Riordan~\cite{rio58}; see the monograph~\cite{FS70} for a comprehensive treatment. In general, the study of descents has become central to algebraic combinatorics for understanding symmetric functions, representation theory~\cite{sag91}, and Coxeter groups~\cite{rei95}.

The \textit{standardization} of a permutation $\pi$ on a set $\{j_1,j_2,\ldots,j_r\}$, where
$j_1<j_2<\cdots<j_r$, is the permutation $\pi'$ obtained by replacing $j_i$ by $i$, for $i=1,2,\ldots,r$. We denote $\mathrm{std}(\pi)$ the standardization of $\pi$. We say a permutation $\pi \in \S_n$ \textit{contains} the pattern $\omega \in \S_k$ if there exists a subsequences of $\pi$ whose standardization is exactly $\omega$, otherwise we say $\pi$ \textit{avoids} the pattern $\omega$.
The set of all permutations of $\S_n$ avoiding $\omega$ is denoted by $\S_n(\omega)$. For the most part of this paper, we use the one-line notation $\pi=\pi_1\pi_2\cdots \pi_n$ to represent a permutation $\pi$, where $\pi_i$ is the image of $i$ under $\pi$ for every $i\in[n]:=\{1,2,\ldots,n\}$.

A \textit{vincular} pattern $\omega $ of length $k$ is a permutation in ${\S}_k$, some of whose consecutive letters may be underlined.
If $\omega $ contains $\underline{\omega_i\omega_{i+1}\cdots \omega_{j}}$, then the letters corresponding to $\omega_i,\omega _{i+1},\ldots,\omega_{j}$ 
in an occurrence of $\omega $ in a permutation must be adjacent as well, whereas there is no adjacency condition for non-underlined consecutive letters.
For example, the pattern $\underline{12}3\underline{45}$ occurs in the permutation $243568179$ four times,
as the subsequences $24568$, $24579$, $24679$, and $35679$.
Patterns in which all letters are underlined are called \textit{consecutive} patterns; see Elizalde's survey~\cite{eliz16} and the references therein for further information on consecutive patterns. Unless otherwise noted, we reserve the greek letter $\sigma$ for a consecutive pattern, 
i.e., $\sigma=\underline{\sigma_1\sigma_2\cdots\sigma_k}$ if the length of $\sigma$ is $k$. Besides consecutive patterns, in this paper we are also concerned with vincular patterns of the form $\sigma j$, where $\std(\sigma)$ is a consecutive pattern of length $k-1$ and $1\le j\le k$. Following \cite{CN18}, such a pattern is called a \textit{quasi-consecutive} pattern.

A function $\st:\S_n \to \mathbb{N} $ is clalled a \textit{permutation statistic}, and the systematic study of permutation
satistics dates back to MacMahon~\cite{mac60}. We are concerned with two classic statistics, the descent number $\des$ and the inversion number $\inv$, whose definitions are recalled below.
\begin{align*} 
\des(\pi) &= |\{i \in [n-1]:\pi_i>\pi_{i+1}\}|,\\
\inv(\pi) &= |\{(i,j)\in [n]^2: i<j \text{ and } \pi_i>\pi_j\}|.
\end{align*}

If $\omega$, $\omega '$ are two patterns and $|\S_n(\omega )|=|\S_n(\omega ')|$ for all $n\ge 1$, then we say $\omega $ and $\omega '$ are \textit{Wilf-equivalent} and write $\omega  \sim \omega '$.
Moreover, for a given permutation statistic ``st", if the following equinumerosity holds for all $n\ge 1$ and $k\ge 0$, 
$$|\{\pi \in \S_n(\omega ): \st(\pi)=k\}|=|\{\pi \in \S_n(\omega '): \st(\pi)=k\}|,$$ 
then we say that $\omega $ and $\omega '$ are \textit{$\st$-Wilf-equivalent} and write $\omega  \overset{\st}{\sim} \omega '$. According to these definitions, one sees that $\omega $ and $\omega '$ may be Wilf-equivalent without being $\st$-Wilf-equivalent. For example, $123$ and $321$ are not $\des$-Wilf-equivalent since $\des(123)=0$ and $\des(321)=2$, even though they are Wilf-equivalent, and both $\S_n(123)$ and $\S_n(321)$ are enumerated by the Catalan numbers, a classic result attributed to MacMahon~\cite{mac60} and Knuth~\cite{knu75}.

Let $A_n^{\omega}(t)=\sum_{\pi \in \S_n(\omega)} t^{\des(\pi)}$ be the generating function of the $\omega$-avoiding permutations $\pi$ of length $n\geq 0$, where $t$ keeps track of the number of descents in $\pi$, and we define
\begin{align*}
A^{\omega}(x,t) &:= \sum_{n\ge 0} A_n^{\omega}(t)\dfrac{x^n}{n!},\\
B^{\omega}(x,t) &:= \sum_{n\ge 0} A_n^{\omega}(t) x^n.
\end{align*} 
Note that $A_0^{\omega}(t)=1$ as we view the empty word as the only permutation of length $0$.

The \textit{reverse} of a permutation $\pi=\pi_1\pi_2\cdots \pi_n$ is the permutation $\pi^{r}=\pi_n\pi_{n-1}\cdots\pi_1$.
The \textit{complement} $\pi^{c}$ of $\pi$ is the permutation $\pi_1'\pi_2'\cdots\pi_n'$ where $\pi_i'=n+1-\pi_i$. By \textit{inverse} we mean the regular group theoretical inverse on permutations, that is, the $\pi_{i}$-th position of the inverse $\pi^{-1}$ in its one-line notation is occupied by $i$. The correspondences between $\pi$ and its reverse $\pi^r$, complement $\pi^c$, inverse $\pi^{-1}$ are called the trivial bijections (actually involutions) from $\S_n$ to itself,
denoted by $\phi _r$, $\phi_c$, and $\phi_{-1}$, respectively.
Also, note that $\phi_c$ and $\phi_r$ are commutative, i.e., $\phi_c\circ\phi_r=\phi_r\circ\phi_c $.
For convenience, we write $\pi^{rc}:=\phi_c(\phi_r(\pi))$.

Let $\sigma$ be a consecutive pattern on $\{1,2,\ldots,k-1\}$ and let the pattern $p=\sigma k$,
so that $p$ is a quasi-consecutive pattern ending with the largest element. For example, if $\sigma=\underline{231} $ then $p=\underline{231}4$. The following theorem links the Wilf-equivalences involving $\sigma$ with those involving $p$, and it was derived by both Elizalde~\cite{eli06} and Kitaev~\cite{kit05}. 
\begin{theorem}\label{thm:eli-kit}
Suppose $\sigma\sim\sigma'$ are two Wilf-equivalent consecutive patterns of length $k-1$, then
$$\sigma k\sim\sigma'k.$$
\end{theorem}

It turns out that a more general result is true, since the descent distribution over $\S_n(\sigma)$ is directly related to the descent distribution over $\S_n(p)$ in the sense of the following theorem. This is the first main result of the paper and we refer to it as the ``Structure Theorem''.

\begin{theorem}[Structure Theorem]\label{structure thm}
    For any consecutive pattern $\sigma$ defining the quasi-consecutive pattern $p=\sigma k$, we have 
    \begin{align}
    \label{id:struc}
    A^p(x,t) = \frac{1}{t} (e^{t \int_{0}^{x} A^{\sigma}(y,t) \,dy}-1) +1.
    \end{align}
\end{theorem}

The rest of this paper is organized as follows. Section~\ref{sec:structure theorem} is devoted to a proof of Theorem~\ref{structure thm}, which is then utilized to compute generating functions and complete the classification for quasi-consecutive patterns of length 3. A similar $\des$-Wilf classification for quasi-consecutive patterns of length 4 is presented in Section~\ref{sec:classifications of 4}. Then, in Section~\ref{sec:g.f. of 4}, two of the generating functions for length 4 quasi-consecutive patterns are derived explicitly via the ``generalized run theorem''. Finally, in Section~\ref{conclusion} we provide concluding remarks.

\section{Structure Theorem and classification of length-3 quasi-consecutive patterns}\label{sec:structure theorem}

\subsection{Proof of Theorem~\ref{structure thm}}
We begin with a proof of Theorem~\ref{structure thm}, and then collect a few preliminary results for later use.

\begin{proof}[Proof of Theorem~\ref{structure thm}]
    Given a permutation $\pi\in \S_{n+1}(p)$, suppose it has the decomposition $\pi=\pi'(n+1)\pi''$ as shown in Fig.~\ref{fig:decomposition of pi(p)}, where the black dot represents the element $n+1$ in $\pi$, while the left square labeled $A^{\sigma}$ refers to the permutation $\pi'$ of
    length $i\,(0\le i\le n)$, and the right square labeled $A^p$ refers to the permutation $\pi''$. The key observation is that $\pi$ is $p$-avoiding if and only if $\pi'$ is $\sigma$-avoiding and $\pi''$ is $p$-avoiding. Consequently, by considering all possible $i$, we have, for $n\ge 0$,
    \begin{equation}\label{eqn-1}
    A_{n+1}^p(t)=t\sum_{i=0}^{n-1}\binom{n}{i}A_i^{\sigma}(t)A_{n-i}^p(t)+A_n^{\sigma}(t).\end{equation}

    \begin{figure}[ht]
        \centering
        \begin{tikzpicture}
            \draw (-6,5) rectangle (-5,4);
            \fill (-4.7,5.3) circle[radius=2pt];
            \node at (-4.1,5.5) {$n+1$};
            \draw (-4.4,5) rectangle (-3.4,4);
            \draw[dashed] (-4.7,4) -- (-4.7,5.3);
            \node at (-5.5,4.5) {$A^{\sigma}$};
            \node at (-3.9,4.5) {$A^p$};
            \node at (-5.5,3.8) {$\underbrace{\phantom{a+b}}_i$ };
        \end{tikzpicture}
        \caption{A decomposition of $\pi \in \S_{n+1}(p)$}
        \label{fig:decomposition of pi(p)}
    \end{figure}
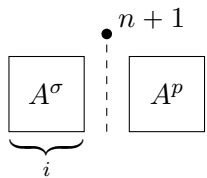

\noindent
Then, multiplying by $\dfrac{x^n}{n!}$ both sides of (\ref{eqn-1}) and summing over all $n\geq 0$, we obtain
    \begin{align*}
        \sum_{n\ge 0}\frac{A_{n+1}^p(t)x^n}{n!} &= t\sum_{n\ge 0}\sum_{i=0}^{n-1}\frac{A_i^{\sigma}(t)x^i}{i!}\frac{A_{n-i}^p(t)x^{n-i}}{(n-i)!}+\sum_{n\ge 0}A_n^{\sigma}(t)\frac{x^n}{n!} \\
        &= t\sum_{n\ge 0}\sum_{i=0}^{n}\frac{A_i^{\sigma}(t)x^i}{i!}\frac{A_{n-i}^p(t)x^{n-i}}{(n-i)!}-t\sum_{n\ge0}A_n^{\sigma}(t)\frac{x^n}{n!}+\sum_{n\ge 0}A_n^{\sigma}(t)\frac{x^n}{n!} \\
        &= tA^{\sigma}(x,t)A^p(x,t) -tA^{\sigma}(x,t)+A^{\sigma}(x,t).
    \end{align*}

    It follows that $$\sum_{n\ge 0}\frac{A_{n+1}^p(t)x^n}{n!}=\frac{{\partial A^p}}{\partial x}=tA^{\sigma}A^p+(1-t)A^{\sigma}=tA^{\sigma}\left(A^p-1+\frac{1}{t}\right),$$
    where we abbreviate $A^{\sigma}(x,t)$ (resp., $A^p(x,t)$) as $A^{\sigma}$ (resp., $A^p$).
    Setting $\widetilde{A^p}=A^{p}-1+\dfrac{1}{t}$, the initial condition becomes $\widetilde{A^p}(0,t)=\dfrac{1}{t}$ since $A^p(0,t)=1$, and 
    we get $\dfrac{\partial \widetilde{A^p}}{\partial x}=tA^{\sigma}\widetilde{A^p}.$
    Solving this equation with the initial condition, we have 
    $\ln \widetilde{A^p}-\ln \dfrac{1}{t}=t \int_{0}^{x} A^{\sigma}(y,t) \,dy$, and 
    $\widetilde{A^p} = \dfrac{1}{t} e^{t \int_{0}^{x} A^{\sigma}(y,t) \,dy}$.
    Therefore, $A^p = \dfrac{1}{t} e^{t \int_{0}^{x} A^{\sigma}(y,t) \,dy} +1-\dfrac{1}{t}=\dfrac{1}{t} (e^{t \int_{0}^{x} A^{\sigma}(y,t) \,dy}-1) +1$.
\end{proof}

The following result immediately follows from the Structure Theorem, and is useful when we consider the $\des$-Wilf-equivalences for various patterns. It also strengthens Theorem~\ref{thm:eli-kit}.

\begin{corollary}\label{col:structure}
    $\sigma\overset{\des}{\sim}\sigma'$ if and only if $\sigma k\overset{\des}{\sim}\sigma' k$, where $\sigma$ and $\sigma'$ are two consecutive patterns on $\{1,2,\ldots,k-1\}$.
\end{corollary}

Note that the three trivial bijections give rise to obvious Wilf-equivalences such as $123\sim 321$. When we consider $\des$-Wilf-equivalences however, the following result will be handy. For any vincular pattern $\omega$, we abuse the notation and still use $\omega^c$ to denote the vincular pattern whose underlying permutation is the complement of the underlying permutation of $\omega$ with the same set of positions being underlined. For instance, $(\underline{231}4 )^c=\underline{324}1$ and $(\underline{21}4\underline{53})^c=\underline{45}2\underline{13}$. The reverse $\omega^r$ and the reverse-complement $\omega^{cr}=\omega^{rc}$ are understood similarly. So, for example,  $(\underline{231}4 )^r=4\underline{132}$ and $(\underline{21}4\underline{53})^{rc}=\underline{31}2\underline{54}$.

\begin{proposition} \label{omega--omega^rc}
    Let $\omega$ be a vincular pattern, then $\omega \overset{\des}{\sim} \omega^{rc}$, or equivalently $\omega^c \overset{\des}{\sim} \omega^{r}$.
\end{proposition}

\begin{proof}
    It is easily seen that $\des(\pi^{rc})=n-1-\des(\pi^r)=n-1-(n-1-\des(\pi))=\des(\pi)$.
    This indicates that the composition map $\phi_c\circ\phi_r$ preserves the permutation statistic $\des$. Moreover, $\pi \in \S_n(\omega )$ if and only if $\pi^{rc} \in \S_n(\omega^{rc} )$.
    Hence $\omega \overset{\des}{\sim} \omega^{rc}$.
\end{proof}

The following fact will facilitate our derivations of generating functions.

\begin{proposition} \label{omega--omega^c}
    Let $\omega$ be a vincular pattern, then $A^{\omega^c}(x,t)=A^{\omega^r}(x,t)=\dfrac{1}{t} (A^{\omega}(xt,\dfrac{1}{t})-1)+1$.
\end{proposition}

\begin{proof}
    Note that $\des(\pi ^c)=n-1-\des(\pi)$ for $n\ge 1$, and $\pi \in \S_n(\omega) $ if and only if $\pi ^c \in \S_n(\omega^c)$.
    It follows that
    \begin{align*}
        A^{\omega^c}(x,t)&=\sum_{n\ge 0} \sum_{\pi\in \S_n(\omega ^c)} t ^{\des (\pi)}\dfrac{x^n}{n!} \\
        &= \sum_{n\ge 0} \sum_{\pi^c \in \S_n(\omega ^c)} t ^{\des (\pi^c)}\dfrac{x^n}{n!}\\
        &=1+ \sum_{n\ge 1} \sum_{\pi \in \S_n(\omega )} t ^{n-1-\des (\pi)}\dfrac{x^n}{n!}  \\
        &=1+\dfrac{1}{t} \left(A^{\omega}(xt,\dfrac{1}{t})-1\right).
    \end{align*}
   By Proposition~\ref{omega--omega^rc}, $\omega^r \overset{\des}{\sim} \omega^{c}$. Hence, $A^{\omega^c}(x,t) = A^{\omega^r}(x,t)$, completing the proof.
\end{proof}


\subsection{Classifications for lengths 2 and 3}

We now provide the complete $\des$-Wilf classification for quasi-consecutive patterns $p = \sigma j$ of length 2 or 3. Using the Structure Theorem, we compute the generating functions $A^p(x, t)$ for all possible choices of $p$, except for $p = \underline{13}2$ and $p = \underline{31}2$. The case of length 2 patterns, $p = \underline{1}2$ and $p = \underline{2}1$, is trivial, as each class contains a unique permutation: $12\cdots n \in \S_n(\underline{2}1)$ and $n(n-1)\cdots 1 \in \S_n(\underline{1}2)$. Their generating functions are straightforward to derive and are included in Table~\ref{tab:2 and 3} for completeness.


\begin{table}[h!]
    \setlength{\tabcolsep}{15pt}    
    \renewcommand{\arraystretch}{1.2}   
    \begin{center}
     \begin{tabular}{c|c|c} 
        Pattern $\sigma j$  & $A^{\sigma j}(x,t)$ or $B^{\sigma j}(x,t)$ & Proved in \\ \hline 
        $\underline{1}2$ & $\dfrac{1}{t}(e^{tx}-1)+1$ & N/A  \\ 
        $\underline{2}1$ & $e^{x}$ & N/A  \\ 
        $\underline{12}3$ & $\dfrac{1}{t} (e^{\frac{1}{t}(e^{tx}-1)+x(t-1)}-1)+1$ & Theorem~\ref{thm:12-3 and 21-3} \\ 
        $\underline{32}1$ & $e^{t(e^{x}-1)+x(1-t)}$ & Corollary~\ref{cor:32-1 and 23-1} \\ 
        $\underline{21}3$ & $\dfrac{1}{t} (e^{t(e^x-1)}-1)+1$ & Theorem~\ref{thm:12-3 and 21-3} \\ 
        $\underline{23}1$ & $e^{\frac{1}{t}(e^{xt}-1)}$ &  Corollary~\ref{cor:32-1 and 23-1} \\ 
        $\underline{13}2$, $\underline{31}2$ & $\dfrac{1+x(t-1)-\sqrt{1-2x(t+1)+x^2(t-1)^2}}{2tx}$ &  Proposition~\ref{13-2}
        \end{tabular}
  \end{center}
  \caption{The $\des$-Wilf classifications of pattern $\sigma j$ whose length is $2$ or $3$}
  \label{tab:2 and 3}
\end{table}





Next, we compute the generating functions for patterns of length 3. Among all six patterns, only half of the derivations (one from each complementary pair) are necessary, thanks to Proposition~\ref{omega--omega^c}.

\begin{theorem} \label{thm:12-3 and 21-3}
    We have
    \begin{align*}
    A^{\underline{12}3}(x,t) &= \dfrac{1}{t} (e^{\frac{1}{t}(e^{tx}-1)+x(t-1)}-1)+1,\\
    A^{\underline{21}3}(x,t) &= \dfrac{1}{t} (e^{t(e^x-1)}-1)+1.
    \end{align*}
\end{theorem}

\begin{proof}
    We first compute $A^{\underline{12}}(x,t)$. Clearly, $\S_n(\underline{12})=\S_n(12)=\S_n(\underline{1}2)=\{n(n-1)\cdots 1\}$. Hence $A^{\underline{12}}(x,t)=A^{\underline{1}2}(x,t)=\dfrac{1}{t}(e^{tx}-1)+1$.
    By Theorem~\ref{structure thm}, we have 
    $\displaystyle{A^{\underline{12}3}(x,t)=\frac{1}{t} (e^{t \int_{0}^{x} A^{\underline{12}}(y,t) \,dy}-1) +1 }$, while 
    \begin{align*}
        \int_{0}^{x} A^{\underline{12}}(y,t) \,dy &=\int_{0}^{x} \left(\frac{1}{t}(e^{ty}-1)+1\right)\,dy \\
        &= \left. \left(\frac{1}{t}(\frac{1}{t} e^{ty}-y)+y\right) \right|_{0}^{x} \\
        &= \frac{1}{t^2}(e^{tx}-1)+\frac{x}{t}(t-1).
    \end{align*}
    We thus obtain $\displaystyle{A^{\underline{12}3}(x,t)=\frac{1}{t} (e^{\frac{1}{t}(e^{tx}-1)+x(t-1)}-1)+1}$, as desired. The derivation of $A^{\underline{21}3}(x,t)$ is similar and is omitted.
\end{proof}

Applying Proposition~\ref{omega--omega^c}, we immediately get the generating functions for two complementary patterns through simple calculation.

\begin{corollary} \label{cor:32-1 and 23-1}
    We have
    \begin{align*}
    A^{\underline{32}1}(x,t) &= e^{t(e^{x}-1)+x(1-t)},\\
    A^{\underline{23}1}(x,t) &= e^{\frac{1}{t}(e^{xt}-1)}.
    \end{align*}
\end{corollary}

\begin{remark}
It is worth noting that when we plug in $t=1$ for each of the four generating functions contained in Theorem~\ref{thm:12-3 and 21-3} and Corollary~\ref{cor:32-1 and 23-1}, they all reduce to $e^{e^x-1}$, which is the well-known exponential generating function for the Bell numbers $\{B_n\}_{n\ge 0}=\{1,1,2,5,15,52,203,877,\ldots\}$; see \href{https://oeis.org/A000110}{\tt{oeis:A000110}}. This enumerative fact was first observed by Claesson~\cite[Prop.~2 and 3]{cla01}. In fact, Claesson constructed two bijections between pattern avoiding permutations and set partitions to establish that
\begin{align}\label{id:perm-setpar}
|\S_n(1\underline{23})|=|\S_n(1\underline{32})|=B_n.
\end{align}
\end{remark}

Taking into consideration the effect of these bijections on the descent numbers of pattern avoiding permutations, we are naturally led to the following refinement of \eqref{id:perm-setpar}. Let $\Par_n$ be the collection of all set partitions of $[n]$. Given a set partition $\Pi\in\Par_n$, we denote by $\blk(\Pi)$ and $\sing(\Pi)$ the number of blocks and the number of singletons (i.e., $1$-blocks) of $\Pi$, respectively. For example, the set partition $\Pi=135/2/4689/7\in\Par_9$ satisfies $\blk(\Pi)=4$ and $\sing(\Pi)=2$.

\begin{corollary}
We have the following alternative interpretations of the generating functions:
\begin{align}
\label{id:gf321 in setpar}
A^{\underline{32}1}(x,t) &= \sum_{n\ge 0}\sum_{\Pi\in\Par_n}t^{\blk(\Pi)-\sing(\Pi)}\frac{x^n}{n!},\\
A^{\underline{21}3}(x,t) &= \sum_{n\ge 0}\sum_{\Pi\in\Par_n}t^{\blk(\Pi)-1}\frac{x^n}{n!}.
\label{id:gf213 in setpar}
\end{align}
\end{corollary}

\begin{proof}
We begin with the proof of \eqref{id:gf321 in setpar}. In view of the original definition of $A^{\underline{32}1}(x,t)$, it suffices to show that the statistic ``$\des$'' over $\S_n(\underline{32}1)$ is equidistributed with ``$\blk -\sing$'' over $\Par_n$. Combining the bijection constructed in Claesson's second proof of \cite[Prop.\ 2]{cla01} with the reversal map $\phi_r$, we get a bijection, say $\varphi$, from $\Par_n$ to $\S_n(\underline{32}1)$ that can be explicitly described as follows. Given a partition $\Pi\in\Par_n$, we say it is expressed in the \textit{standard representation} if the following conditions are enforced:
\begin{enumerate}[label=(\roman*)]
    \item Each block ends with its smallest element, and all remaining elements are written in increasing order.
    \item The blocks are written in increasing order of their smallest element with slashes separating the blocks.
\end{enumerate}
The standard representation for the previous set partition $\Pi$ is seen to be $351/2/6894/7$. Next, we simply remove all the slashes to get its image permutation $\pi=\varphi(\Pi)$. Clearly, the descents of $\pi$ are precisely between the penultimate and the ending elements within each non-singleton block of $\Pi$ (in its standard representation). This not only implies that $\pi$ is $\underline{32}1$-avoiding, but also reveals the relation $\des(\pi)=\blk(\Pi)-\sing(\Pi)$. To see that $\varphi$ is indeed invertible, note that one can recover the preimage $\varphi^{-1}(\pi)$ in its standard representation by inserting a slash after every right-to-left minimum of $\pi$. 

The proof of \eqref{id:gf213 in setpar} is in the same vein but with the following conditions on \textit{another standard representation} of set partitions:
\begin{enumerate}[label=(\roman*)]
    \item The elements in each block are written in increasing order.
    \item The blocks are written in decreasing order of their last (also largest) element with slashes separating the blocks. \qedhere
\end{enumerate}
\end{proof}

For the remaining pair of patterns $\underline{13}2$ and $\underline{31}2 = (\underline{13}2)^c$, the Structure Theorem is inapplicable, as neither of them ends with the largest element. Moreover, it was already noted in \cite[Lem.\ 2]{cla01} (see also \cite[Lem.\ 2.8]{FTHZ19}) that \begin{align*} \S_n(\underline{13}2) = \S_n(132), \quad \S_n(\underline{31}2) = \S_n(312), \ \S_n(2\underline{13}) = \S_n(213), \quad \S_n(2\underline{31}) = \S_n(231). \end{align*} Hence, all of them are enumerated by the {\em Catalan numbers} $C_n = \frac{1}{n+1}\binom{2n}{n}$. Based on these observations, we instead consider the ordinary generating function $B^{\omega}(x,t) = \sum_{n \geq 0} \sum_{\pi \in \S_n(\omega)} t^{\des(\pi)} x^n$ for the last two patterns. The relation in Proposition~\ref{omega--omega^c} still holds, with $A^{\omega}(x,t)$ replaced by $B^{\omega}(x,t)$.

\begin{proposition} \label{13-2}
    We have
    \begin{align*}
    B^{\underline{13}2}(x,t) = B^{\underline{31}2}(x,t)= \dfrac{1+x(t-1)-\sqrt{1-2x(t+1)+x^2(t-1)^2}}{2tx}.
    \end{align*}
\end{proposition}

\begin{proof}
By Proposition~\ref{omega--omega^rc}, we know that $\underline{31}2 \overset{\des}{\sim} 2\underline{31}$. Now the descent distribution over $\S_n(2\underline{31})=\S_n(231)$ generates the {\em Narayana polynomial} $N_n(t):=\sum_{0\le k\le n-1}N_{n,k}t^k$, where
$$N_{n,k}:=|\{\pi\in \S_n(231):\des(\pi)=k\}|=\frac{1}{k+1}\binom{n}{k}\binom{n-1}{k}$$
is the \textit{Narayana number}. The ordinary generating function of Narayana polynomials $N_n(t)$ is known to be (for instance, see \cite[Sect.~2.3]{pet15})
$$B^{231}(x,t)=\dfrac{1+x(t-1)-\sqrt{1-2x(t+1)+x^2(t-1)^2}}{2tx},$$
which gives the desired expression for $B^{\underline{31}2}(x,t)$. To derive the same formula for $B^{\underline{13}2}(x,t)$, we simply apply Proposition~\ref{omega--omega^c} and perform the derivations. Alternatively, it suffices to note the palindromicity of Narayana numbers (i.e., $N_{n,k}=N_{n,n-1-k}$) and the relation $\des(\pi^r)=n-1-\des(\pi)$ that holds for all $\pi\in\S_n$.
\end{proof}


\section{des-Wilf classification of length 4 quasi-consecutive patterns} \label{sec:classifications of 4}

In this section, we present a complete $\des$-Wilf equivalence classification for quasi-consecutive patterns of length four. As a motivational background, we comment briefly on the Wilf equivalence classification for quasi-consecutive patterns of length four. In total there are $4!=24$ patterns. Upon taking complement, only half of them need to be considered. It was already conceived by Baxter and Pudwell \cite[Table 6]{BP12} that these twelve patterns split into five Wilf equivalent classes. Their enumeration sequences have all been registered on the OEIS (see the last column in Table~\ref{tab: classifications of length 4}). Many of these equivalences could be easily deduced from Theorem~\ref{thm:eli-kit}, with the remaining ones conjectured in \cite[Conj.~17]{BP12} and later confirmed by Baxter and Shattuck~\cite{BS15}. In fact, almost all length four vincular patterns (not only the quasi-consecutive ones) have been classified according to the Wilf equivalence, with two pairs
$$\underline{23}14\sim 1\underline{23}4,\text{ and } \underline{14}23\sim 2\underline{14}3$$
raised as a conjecture in \cite{BS15}. Solving them would then complete the Wilf equivalence classification for all length four vincular patterns.

When we take the descent number into consideration, numerical data (see Table~\ref{tab: descent vector of length 4}) seems to suggest that there are fourteen different classes of quasi-consecutive patterns. The \textit{descent vector} of pattern $\omega$ at level $n$ is defined as $(a_{n,0}^{\omega},a_{n,1}^{\omega},\ldots,a_{n,n-1}^{\omega})$
where $a_{n,k}^{\omega}=|\{\pi\in \S_n(\omega):\des(\pi)=k\}|$. The different descent vectors shown in the third column of Table~\ref{tab: descent vector of length 4} make it evident that there should be at least fourteen distinctive classes. To show that there are indeed fourteen classes, we are going to construct descent preserving bijections for the patterns contained in the same class, ranging from class No.~1 to class No.~6. For the singleton patterns $\underline{123}4$ and $\underline{321}4$, we apply the generalized run theorem (see Theorem~\ref{thm:gen run thm}) to compute their generating functions in the next section. All these results are summarized in Table~\ref{tab: classifications of length 4}. Also note that due to Proposition~\ref{omega--omega^c}, only half of the patterns need to be considered (one out of each complementary pair). So for instance, we shall construct a bijection between $\S_n(\underline{134}2)$ and $\S_n(\underline{124}3)$ to explain the $\des$-Wilf equivalence between these two patterns in class No.~5, but omit the result for the two patterns in class No.~6 since $\underline{431}2=(\underline{124}3)^c$ and $\underline{421}3=(\underline{134}2)^c$. This also explains why we only include seven classes in Table~\ref{tab: classifications of length 4}.

\begin{table}[h!]
    \setlength{\tabcolsep}{15pt}
    \renewcommand{\arraystretch}{1.2}
    \begin{center}
      \begin{tabular}{c|c|c} 
        Class No. & Pattern $\sigma j$ & Descent vector of $\sigma j$ for $n=7$ \\
        \hline
        1 & \underline{312}4,\ \underline{231}4,\ \underline{241}3  & 1,\ 78,\ 724,\ 1706,\ 1041,\ 120,\ 1 \\
        2 & \underline{213}4,\ \underline{132}4,\ \underline{142}3  & 1,\ 110,\ 894,\ 1648,\ 897,\ 120,\ 1 \\
        3 & \underline{243}1,\ \underline{324}1,\ \underline{314}2  & 1,\ 120,\ 1041,\ 1706,\ 724,\ 78,\ 1 \\
        4 & \underline{342}1,\ \underline{423}1,\ \underline{413}2  & 1,\ 120,\ 897,\ 1648,\ 894,\ 110,\ 1 \\
        5 & \underline{134}2,\ \underline{124}3  & 1,\ 56,\ 637,\ 1756,\ 1089,\ 120,\ 1 \\
        6 & \underline{431}2,\ \underline{421}3  & 1,\ 120,\ 1089,\ 1756,\ 637,\ 56,\ 1 \\ 
        7 & $\underline{123}4$ & 0,\ 0,\ 481,\ 2022,\ 1191,\ 120,\ 1  \\ 
        8 & $\underline{321}4$ & 1,\ 120,\ 1119,\ 1853,\ 665,\ 56,\ 1   \\
        9 & $\underline{214}3$ & 1,\ 120,\ 1080,\ 1740,\ 639,\ 56,\ 1   \\ 
        10 & $\underline{412}3$ & 1,\ 56,\ 632,\ 1732,\ 1080,\ 120,\ 1   \\
        11 & $\underline{143}2$ & 1,\ 120,\ 1080,\ 1732,\ 632,\ 56,\ 1   \\ 
        12 & $\underline{341}2$ & 1,\ 56,\ 639,\ 1740,\ 1080,\ 120,\ 1   \\ 
        13 & $\underline{234}1$ & 1,\ 56,\ 665,\ 1853,\ 1119,\ 120,\ 1   \\
        14 & $\underline{432}1$ & 1,\ 120,\ 1191,\ 2022,\ 481,\ 0,\ 0   
      \end{tabular}
    \end{center}
  \caption{Fourteen $\des$-Wilf equivalence classes with their descent vectors for $n=7$}
  \label{tab: descent vector of length 4}
\end{table}

For certain given patterns $\alpha$ and $\beta$, we divide $\S_n$ into four disjoint subsets:
\begin{align*}
  \S_n(\alpha,\beta)&:=\S_n(\alpha) \bigcap \S_n(\beta), \\
  \S_n(\alpha,\check{\beta})&:=\{\pi \in \S_n(\alpha): \beta\text{ occurs in }\pi\} ,\\
  \S_n(\check{\alpha},\beta)&:=\{\pi \in \S_n(\beta): \alpha\text{ occurs in }\pi\} ,\\
  \S_n(\check{\alpha},\check{\beta})&:=\{\pi \in \S_n:\text{both }\alpha\text{ and }\beta\text{ occur in }\pi\} . 
\end{align*}

When constructing a bijection between, say $\S_n(\alpha)$ and $\S_n(\beta)$, we usually set the permutations in $\S_n(\alpha,\beta)$ to be fixed points, and try to find a $\des$-preserving bijection between $\S_n(\alpha,\check{\beta})$ and $\S_n(\check{\alpha},\beta)$, then $\alpha \overset{\des}{\sim }\beta$ follows immediately.

We note that a bijective proof of $\underline{312}4 \overset{\des}{\sim }\underline{231}4\overset{\des}{\sim }\underline{241}3$ has already appeared in \cite{BDGZ19}. It was proved via bijections constructed on the complement sets of pattern avoiding permutations. Their proof of $\underline{312}4 \overset{\des}{\sim }\underline{231}4$ goes roughly as follows: if $\pi \in \S_n(\widecheck{\underline{312}4},\widecheck{\underline{231}4})$,
then $\pi$ is a fixed point; otherwise they find a special occurrence of pattern $\underline{312}4$ in $\pi$, change it to an occurrence of $\underline{231}4$, claiming that this results in a bijection from $\S_n(\widecheck{\underline{312}4},\underline{231}4)$ to $\S_n(\underline{312}4,\widecheck{\underline{231}4})$. But they overlooked the case when there are two or more occurrences of pattern $\underline{312}4$ in $\pi$. In that case, via the bijection, say $\phi$ in \cite{BDGZ19}, $\phi(\pi)$ actually falls in $\S_n(\widecheck{\underline{312}4},\widecheck{\underline{231}4})$. Hence the above proof found in \cite{BDGZ19} is flawed. Their proof of $\underline{231}4\overset{\des}{\sim }\underline{241}3$ is incorrect for the same reason. Moreover, it is worth noting that the Wilf equivalences $\underline{231}4\sim\underline{241}3$, $\underline{132}4\sim\underline{142}3$, and $\underline{134}2\sim\underline{124}3$ are covered by Baxter and Shattuck's general equivalence results in \cite[Thm.~6 and Thm.~9]{BS15}, and it was remarked there that their proof actually shows the stronger $\des$-Wilf equivalences. Nevertheless, comparing to their inductive approach, our bijective proofs below are described via certain algorithms and seem to be more explicit.

Before we get into the bijective proofs, we would like to point out that $\underline{312} \overset{\des}{\sim }\underline{231} $ by Proposition~\ref{omega--omega^rc}, then we can apply Corollary~\ref{col:structure} to deduce that $\underline{312}4 \overset{\des}{\sim }\underline{231}4$. A similar argument leads to $\underline{213}4 \overset{\des}{\sim }\underline{132}4$. Alternatively, we present the following bijective approach. 

\begin{table}[h!]
\setlength{\tabcolsep}{9pt}  
\renewcommand{\arraystretch}{1.2}  
 \begin{center}
 \begin{tabular}{c|c|c|c} 
 Pattern $\sigma j$ & Bijections & $A^{\sigma}(x,t)$ & OEIS entry\\ \hline
 \underline{312}4,\ \underline{231}4,\ \underline{241}3 & Theorems~\ref{312-4--231-4} and \ref{231-4--241-3} & Corollary~\ref{231} & A071075 \\
 \underline{213}4,\ \underline{132}4,\ \underline{142}3 & Theorems~\ref{213-4--132-4} and \ref{132-4--142-3} & Proposition~\ref{132} & A071075\\
 \underline{134}2,\ \underline{124}3 & Theorem~\ref{134-2--124-3} & ? & A200403 \\
 $\underline{123}4$ & N/A & \eqref{Sn(123)des} & A071076 \\ 
 $\underline{321}4$ & N/A & \eqref{Sn321des} & A071076 \\
 $\underline{214}3$ & N/A & ? & A200405 \\ 
 $\underline{412}3$ & N/A & ? & A200404
 \end{tabular}
 \end{center}
 \caption{Classifications of $\sigma j$ of length $4$}
 \label{tab: classifications of length 4}
\end{table}

\begin{theorem}\label{312-4--231-4}
  We have the $\des$-Wilf equivalence $\underline{312}4 \overset{\des}{\sim }\underline{231}4$. 
\end{theorem}

\begin{proof}
  It suffices to construct a $\des$-preserving bijection 
  $$f: \S_n(\widecheck{\underline{312}4},\underline{231}4) \to \S_n(\underline{312}4,\widecheck{\underline{231}4}).$$
  Suppose $\pi\in \S_n(\widecheck{\underline{312}4},\underline{231}4)$. 
  We first find all occurrences of $\underline{312}4$ in $\pi$, and
  let $A^{\pi}$ be the set of the elements playing the role of $``4"$ in all occurrences of $\underline{312}4$ in $\pi$. Select the largest element in $A^{\pi}$, say $\pi_{j_1}$ for a certain $j_1$.
  Next find all occurrences of $\underline{312}4$ in $\pi^1:=\pi_{j_1+1} \pi_{j_1+2}\cdots \pi_n$ if any. Let $A^{\pi^1}$ be the set of the elements contained in $\pi^1$ playing the role of $``4"$ in any occurrence of $\underline{312}4$ in $\pi^1$, select the largest element in $A^{\pi^1}$, say $\pi_{j_2}$ for a certain $j_2$. Keep doing this until we arrive at a certain $\pi_{j_l}$ so that the suffix $\pi_{j_l+1}\pi_{j_l+2}\cdots \pi_{n}$ is $\underline{312}4$-avoiding. It follows that $\pi_{j_1}>\pi_{j_2}>\cdots>\pi_{j_l}$ and $j_1<j_2<\cdots<j_l$.
  For instance, if $\pi=867953124$, then $\pi_4=9$ and $\pi_9=4$ would be selected with $j_1=4$ and $j_2=9$.

  Next for all $k,\; 1\le k\le l$, we find $\pi_{i_k}$, which is the rightmost element to the left of $\pi_{j_k}$ such that $\pi_{i_k}>\pi_{j_k}$. For $1<k\le l$, such a $\pi_{i_k}$ always exists since at least we have $\pi_{j_{k-1}}>\pi_{j_k}$. If such a $\pi_{i_1}$ does not exist, then we set $i_1=0$ and $\pi_0=\infty$. This happens when $\pi_{j_1}$ is a left-to-right maximum. 
  Figure~\ref{fig:312-4} illustrates the relations of elements in $\pi$. 
  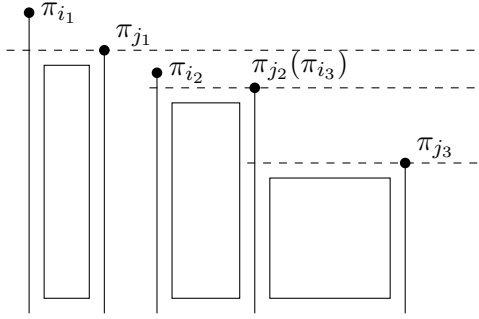
\begin{figure}[ht]
    \centering
    \begin{tikzpicture}
      \draw (0,0)--(0,4);
      \fill (0,4) circle[radius=2pt];
      \node at (0.4,4) {$\pi_{i_1}$};
      \draw (1,0)--(1,3.5);
      \fill (1,3.5) circle[radius=2pt];
      \node at (1.4,3.7) {$\pi_{j_1}$};
      \draw[dashed] (-0.3,3.5) -- (6,3.5);
      \draw (0.2,3.3) rectangle (0.8,0.2);

      \draw (1.7,0)--(1.7,3.2);
      \fill (1.7,3.2) circle[radius=2pt];
      \node at (2.1,3.2) {$\pi_{i_2}$};
      \draw (3,0)--(3,3);
      \fill (3,3) circle[radius=2pt];
      \node at (3.6,3.3) {$\pi_{j_2}(\pi_{i_3})$};
      \draw[dashed] (1.6,3) -- (6,3);
      \draw (1.9,2.8) rectangle (2.8,0.2);

      \draw (5,0)--(5,2);
      \fill (5,2) circle[radius=2pt];
      \node at (5.4,2.2) {$\pi_{j_3}$};
      \draw[dashed] (2.9,2) -- (6,2);
      \draw (3.2,1.8) rectangle (4.8,0.2);
    \end{tikzpicture}
    \caption{An illustration of choosing $\pi_{j_k}$ and $\pi_{i_k}$, $1\le k\le l$}
    \label{fig:312-4}
  \end{figure}

  The choices of $\pi_{i_k},\pi_{j_k}(1\le k \le l)$ imply that,  
  for all $1\le k \le l$, 
  \begin{enumerate}[label=(\roman*)]
      \item there is no $\underline{231}4$ pattern in $\pi_1\pi_2\cdots \pi_{n}$;
      \item $\pi_{i_k+1},\pi_{i_k+2},\cdots ,\pi_{j_k-1}<\pi_{j_k}<\pi_{i_k}$;
      \item there is no $\underline{312}4$ pattern in $\pi_{j_{k-1}+1}\pi_{j_{k-1}+2}\cdots \pi_{i_{k}}(j_{0}:=0)$;
      \item there must be a $\underline{312}4$ pattern in $\pi_{i_k+1} \pi_{i_k+2}\cdots \pi_{j_k}$;
      \item $j_k-i_k\ge 4$.
  \end{enumerate}
  Now we obtain the image permutation $f(\pi)$ from $\pi$ by keeping the positions and values of both $\pi_{i_k}$ and $\pi_{j_k}$, while transforming the subword $\pi_{i_k+1}\pi_{i_k+2}\cdots \pi_{j_{k}-1}$ to $\pi'_{i_k+1}\pi'_{i_k+2}\cdots \pi'_{j_k-1}$, for all $1\le k\le l$, where
  $\pi'_{i_k+1}\pi'_{i_k+2}\cdots \pi'_{j_k-1}$ is the permutation on
  $\{\pi_{i_k+1},\pi_{i_k+2},\cdots, \pi_{j_k-1}\}$ such that
  $$\std(\pi'_{i_k+1}\pi'_{i_k+2}\cdots \pi'_{j_k-1})=(\std (\pi_{i_k+1}\pi_{i_k+2}\cdots \pi_{j_k-1}))^{rc}.$$
  For instance, $f(867953124)=786952314$. Moreover, one should be able to verify the following facts for all $1\le k \le l$.
  \begin{enumerate}[label=(\roman*)]
      \item $\des(\pi)=\des(f(\pi))$ since $\pi_{i_k}$ and $\pi_{j_k}$ are greater than 
      the elements in \\ $\{\pi_{i_k+1},\pi_{i_k+2},\cdots, \pi_{j_k-1}\}$, and the composition $\phi_r\circ\phi_c$ preserves the descent number.
      \item all occurrences of $\underline{312}4$ in $\pi_{i_k+1} \pi_{i_k+2}\cdots \pi_{j_k}$ become occurrences of $\underline{231}4$ in $\pi_{i_k+1}' \pi_{i_k+2}'\cdots \pi_{j_k}'$.
      \item there is no $\underline{312}4$ pattern in $f(\pi)$.
      \item $\pi_{j_k}=f(\pi)_{j_k}$ is also the largest element of $B^{f(\pi)^{k-1}}$, where
      $B^{f(\pi)^{k-1}}$ is        
      the set of elements playing $``4"$ in each occurrence of $\underline{231}4$ in
      $$f(\pi)_{j_{k-1}+1}f(\pi)_{j_{k-1}+2}\cdots f(\pi)_{i_k} f(\pi)_{i_k+1}\cdots f(\pi)_{j_k-1} f(\pi)_{j_k}\cdots f(\pi)_n,\quad (f(\pi)^0:=f(\pi)).
      $$
      Moreover, $A^{\pi^k}=B^{f(\pi)^{k}}$ for all $0\le k \le l-1$ $(\pi^0=\pi,f(\pi)^{0}=f(\pi))$.
      \item there is no $\underline{231}4$ pattern in $\pi_{j_{k-1}+1}\pi_{j_{k-1}+2}\cdots \pi_{i_{k}}$.
  \end{enumerate}
  Therefore, $f(\pi)\in \S_n(\underline{312}4,\widecheck{\underline{231}4})$.

  To show that $f$ is invertible, the key thing to notice is that $f$ keeps the ``maximality" of $\pi_{j_k}(1\le k \le l)$. The inverse mapping $f^{-1}$ is constructed similarly as $f$, except now finding $\pi_{j_k}$ and $\pi_{i_k}$ is with respect to the pattern $\underline{231}4$,
  and we implement the same transformation on the subword $\pi_{i_k+1}\pi_{i_k+2}\cdots \pi_{j_{k}-1}$.
\end{proof}

\begin{remark}
  The bijection $f$ constructed in Theorem~\ref{312-4--231-4} also preserves the statistic $\inv $,
  i.e., $\inv (\pi)=\inv (f(\pi))$ for $\pi \in \S_n(\widecheck{\underline{312}4},\underline{231}4)$ since the composition $\phi_r\circ\phi_c$ preserves the number (and positions) of the inversion pairs, although their values are usually changed.
\end{remark}

We introduce two algorithms here to facilitate our proofs of the other equivalence classes.
Given a letter $a$ and a word $w$ so that $a$ is not contained in $w$.
The first algorithm \textit{REPLACEMENT-I} proceeds as follows.
\begin{enumerate}[label=(\arabic*)]
    \item Start with a pair $(a,w)$.
    \item \textbf{While} $a$ is larger than some element in $w$, \textbf{do} \label{Re2}
    \begin{enumerate}[label=(2.\arabic*)]
        \item Let $y$ be the largest element in $w$ smaller than $a$ and replace $y$ by $a$, get new $w$. \label{Re2.1}
        \item Set $a:=y$.
    \end{enumerate}
    \item Get new pair $(a,w)$ and \textbf{stop}.
\end{enumerate}
The \textit{REPLACEMENT-II} algorithm proceeds as follows.
\begin{enumerate}[label=(\arabic*)]
    \item Start with a pair $(a,w)$.
    \item \textbf{While} $a$ is smaller than some element in $w$, \textbf{do} \label{Re2-}
    \begin{enumerate}[label=(2.\arabic*)]
        \item Let $y$ be the smallest element in $w$ larger than $a$ and replace $y$ by $a$, get new $w$. \label{Re2.1-}
        \item Set $a:=y$.
    \end{enumerate}
    \item Get new pair $(a,w)$ and \textbf{stop}.
\end{enumerate}

\begin{proposition}
    When the pair $(a,w)$ is inputted, let $(a',w')$ and $(a'',w'')$ be the output pair after applying the \textit{REPLACEMENT-I} and \textit{REPLACEMENT-II} algorithms respectively, then 
    \begin{enumerate} 
        \item $a'$ is smaller than every element of $w'$, $a''$ is larger than every element of $w''$; 
        \item $\std(w')=\std(w)=\std(w'')$;
        \item $\des(w')=\des(w)=\des(w'')$.
    \end{enumerate}
\end{proposition}

\begin{proof}
  If $a'$ is larger than any element of $w'$, the step~\ref{Re2} of \textit{REPLACEMENT-I} needs to run for another round before the algorithm terminates, this proves the first property for $(a',w')$. The proofs of other properties should be straightforward and we omit them.
\end{proof}

\begin{theorem}
  We have the $\des$-Wilf equivalence $\underline{231}4 \overset{\des}{\sim }\underline{241}3$. \label{231-4--241-3}
\end{theorem}

\begin{proof}
  We construct a $\des$-preserving bijection
  $$g: \S_n(\underline{231}4,\widecheck{\underline{241}3}) \to \S_n(\widecheck{\underline{231}4},\underline{241}3).$$
  Suppose $\pi:=a_1a_2\cdots a_n\in \S_n(\underline{231}4,\widecheck{\underline{241}3})$. 
  Find all occurrences of $\underline{241}3$ in $\pi$, let $\{j_1,j_2,\ldots,j_l\}$
  be its indices for some $l$, i.e., $\std(a_{j_i}a_{j_i+1}a_{j_i+2})=231$ and there exists $p_i>j_i+2$
  such that $\std(a_{j_i}a_{j_i+1}a_{j_i+2}a_{p_i})=2413$ for all $1\le i \le l$.
  Without loss of generality, set $j_1<j_2<\cdots<j_l$, then $j_{i+1}-j_i\ge 2$ for all $1\le i < l$.
  We notice that since $\pi \in \mathfrak{S}_n(\underline{231}4)$, all of the elements after $a_{j_i+1}$ are smaller than $a_{j_i+1}$ for all $1\le i \le l$. In particular, $a_{j_1+1}>a_{j_2+1}>\cdots>a_{j_l+1}$.
    
  We execute the \textit{REPLACEMENT-I} algorithm on $(a_{j_1+1},a_{j_1+3}a_{j_1+4}\cdots a_n)$, but
  \begin{itemize}
      \item \textit{REPLACEMENT-I}~\ref{Re2} is modified to be ``$a$ is larger than some element which is larger than $a_{j_1}$ in $w$'';
      \item  and add the condition ``$y>a_{j_1}$'' to \textit{REPLACEMENT-I}~\ref{Re2.1}. 
  \end{itemize}
  Denote $\pi^1:=a_1^1\cdots a_{j_1}^1a_{j_1+1}^1a_{j_1+2}^1a_{j_1+3}^1 \cdots a_n^1$, where $(a_{j_1+1}^1,a_{j_1+3}^1a_{j_1+4}^1\cdots a_n^1)$ is the output from the above modified \textit{REPLACEMENT-I} algorithm while all remaining entries are the same as in $\pi$. One verifies that 
  \begin{enumerate}[label=(\roman*)]
      \item $\std(a_{j_1}^1a_{j_1+1}^1a_{j_1+2}^1a_{p_1}^1)=2314$ for some $p_1>j_1+2$;
      \item $\std(a_{j_1+3}^1 \cdots a_n^1)=\std(a_{j_1+3} \cdots a_n)$;
      \item $\des(\pi)=\des(\pi^1)$;
      \item there exists no $\underline{231}4$ pattern before $a_{j_1}^1$ of $\pi^1$;
      \item $a_{j_1+1}^1$ is the smallest element in $a_{j_1+3}a_{j_1+4}\cdots a_n$ such that $a_{j_1}<a_{j_1+1}^1<a_{j_1+1}$;
      \item there exists no $p_1>j_1+2$ such that $\std(a_{j_1}^1a_{j_1+1}^1a_{j_1+2}^1a_{p_1}^1)=2413$;
      \item $\std(a_{j_i}^1a_{j_i+1}^1a_{j_i+2}^1)=231$ and there exists $p_i>j_i+2$ such that $\std(a_{j_i}^1a_{j_i+1}^1a_{j_i+2}^1a_{p_i}^1)=2413$ for all $2\le i\le l$.
  \end{enumerate}
  Next, consider $a_{j_2}^1 a_{j_2+1}^1 a_{j_2+2}^1$, repeat the above operations by simply replacing the subscript $j_1$ with $j_2$ and the 
  superscript $1$ with $2$. Keep doing this for $a_{j_i}^{i-1} a_{j_i+1}^{i-1} a_{j_i+2}^{i-1}$ with $i=2,3,\ldots, l$ in that order. Set the image $g(\pi):=\pi^l$, the final output permutation after applying $l$ times of modified \textit{REPLACEMENT-I} algorithm. We see that there exists no $\underline{241}3$ patterns in $g(\pi)$. More precisely, all occurrences of $\underline{241}3$ in $\pi$ become occurrences of $\underline{231}4$ in $g(\pi)$.

  For example, $\underline{362}514 \mapsto 34\underline{261}5 \mapsto 342516$, and
  $\underline{462}513 \mapsto 45\underline{261}3 \mapsto 452316$. So $g(362514)=342516$, and $g(462513)=452316$.

  The construction of the inverse mapping $g^{-1}$ is straightforward. Observe that for all occurrences of $\underline{231}4$ in $g(\pi):=b_1b_2\cdots b_n$, its set of indices is still given by $\{j_1,j_2,\ldots,j_l\}$. That is to say, $\std(b_{j_i}b_{j_i+1}b_{j_i+2})=231$ and there exists $p_i>j_i+2$
  such that $\std(b_{j_i}b_{j_i+1}b_{j_i+2}b_{p_i})=2314$ for all $1\le i \le l$.
  We execute the suitably modified \textit{REPLACEMENT-II} algorithm on $(b_{j_l+1},b_{j_l+3}b_{j_l+4}\cdots b_n)$ to generate $(b_{j_l+1}^1,b_{j_l+3}^1b_{j_l+4}^1\cdots b_n^1)$, turning the $\underline{231}4$ occurrence beginning with $b_{j_l}b_{j_l+1}b_{j_l+2}$ into a $\underline{241}3$ occurrence beginning with $b_{j_l}b_{j_l+1}^1b_{j_l+2}$. Continue the process to replace the $\underline{231}4$ occurrences that begin with positions $j_{l-1},j_{l-2},\ldots,j_1$, in that order, until we finally arrive at a permutation that is in $\S_n(\underline{231}4,\widecheck{\underline{241}3})$, and it is taken to be the preimage $g^{-1}(b_1b_2\cdots b_n)$. For example, $34\underline{251}6\mapsto  \underline{342}615 \mapsto 362514 $, and $45\underline{231}6 \mapsto  \underline{452}613 \mapsto 462513$.
\end{proof}

Table~\ref{tab:231-4--241-3} lists the example of $\underline{231}4\overset{\des}{\sim} \underline{241}3$ for $n=5$.
The permutations in the same line are in correspondence with each other under our bijection $g$.

\begin{table}[h]
\setlength{\tabcolsep}{15pt}
  \renewcommand{\arraystretch}{1.2}
  \centering
  \begin{tabular}{ccc} 
      &  $\S_n(\underline{231}4,\widecheck{\underline{241}3})$ & $\S_n(\widecheck{\underline{231}4},\underline{241}3)$ \\ \hline
       & $\underline{251}34$ & $\underline{231}45$  \\
   $\mathrm{des}=1$    & $\underline{351}24$ & $\underline{341}25$  \\
       & $2\underline{351}4$ & $2\underline{341}5$  \\
       & $1\underline{352}4$ & $1\underline{342}5$  \\\hline
       & $5\underline{241}3$ & $5\underline{231}4$   \\
       & $\underline{251}43$ &  $\underline{231}54$ \\
       &  $3\underline{251}4$ &  $3\underline{241}5$ \\
        $\mathrm{des}=2$ & $4\underline{251}3$ &  $4\underline{231}5$ \\
       & $\underline{351}42$ & $\underline{341}52$  \\
       & $\underline{352}14$ & $\underline{342}15$  \\
       & $\underline{352}41$ & $\underline{342}51$  \\      
       \hline
  \end{tabular}
  \caption{$\underline{231}4\overset{\des}{\sim} \underline{241}3$ ($n=5$)}
  \label{tab:231-4--241-3}
\end{table}

\begin{theorem} \label{213-4--132-4}
  We have the $\des$-Wilf equivalence $\underline{213}4\overset{\des}{\sim }\underline{132}4$.
\end{theorem}

\begin{proof}
  We can prove $\underline{213}4\overset{\des}{\sim }\underline{132}4$ similarly using 
  the bijection $f$ from Theorem~\ref{312-4--231-4}, because ``4'' is the largest element in pattern
  $\underline{213}4$ and $(213)^{rc}=132$.
\end{proof}

\begin{theorem}
  We have the $\des$-Wilf equivalence $\underline{132}4\overset{\des}{\sim }\underline{142}3$. \label{132-4--142-3}
\end{theorem}

\begin{proof}
  The proof is analogous to that of Theorem~\ref{231-4--241-3}, we only need to transform all of the $\underline{132}4$ occurrences to $\underline{142}3$ occurrences. Note the following modifications to the \textit{REPLACEMENT-I} algorithm though:
  \begin{itemize}
      \item  the step~\ref{Re2} of \textit{REPLACEMENT-I} is modified to ``$a$ is larger than some element which is larger than $a_{j_1+2}$ in $w$'';
      \item  and add the condition ``$y>a_{j_1+2}$'' to \textit{REPLACEMENT-I}~\ref{Re2.1}. 
  \end{itemize}
  The inverse mapping and the \textit{REPLACEMENT-II} algorithm should be adjusted accordingly. We omit the details.
\end{proof}

\begin{theorem}
  We have the $\des$-Wilf equivalence $\underline{134}2 \overset{\des}{\sim } \underline{124}3$.  \label{134-2--124-3}   
\end{theorem}

\begin{proof}
  The same approach for proving Theorem~\ref{231-4--241-3} works for this pair as well. We want to transform all $\underline{134}2$ occurrences to $\underline{124}3$ occurrences, with the two algorithms modified as follows. For the forward mapping relying on \textit{REPLACEMENT-I},
  \begin{itemize}
      \item  \textit{REPLACEMENT-I}~\ref{Re2} is modified to ``$a$ is larger than some element which is larger than $a_{j_1}$ in $w$'';
      \item  and add the condition ``$y>a_{j_1}$'' to \textit{REPLACEMENT-I}~\ref{Re2.1}. 
  \end{itemize}
  For the backward mapping relying on \textit{REPLACEMENT-II}, 
  \begin{itemize}
      \item \textit{REPLACEMENT-II}~\ref{Re2-} is modified to ``$a$ is smaller than some element which is smaller than $a_{j_1+2}$ in $w$'';
      \item  and add the condition ``$y<a_{j_1+2}$'' to \textit{REPLACEMENT-II}~\ref{Re2.1-}. \qedhere
  \end{itemize}
\end{proof}

At this point, we have fully confirmed the fourteen $\des$-Wilf equivalence classes for quasi-consecutive patterns of length four; see Table~\ref{tab: descent vector of length 4}.

\section{Generating functions} \label{sec:g.f. of 4}

According to Theorem~\ref{structure thm}, we have the generating function for the quasi-consecutive pattern $\sigma k$, where $k$ is the largest element, provided the generating function for $\sigma$ is known. In this section, we derive an explicit generating function for one of the patterns, namely $\underline{123}$, and subsequently for $\underline{321}$ as well, thanks to Proposition~\ref{omega--omega^c}. The approach we use is the so-called generalized run theorem; see Theorem~\ref{thm:gen run thm}. For the remaining four patterns, $\underline{132}$, $\underline{213}$, $\underline{231}$, and $\underline{312}$, we demonstrate that their generating functions satisfy certain partial differential equations, as presented in Subsection~\ref{subsec:g.f.132}.

\subsection{The run theorem} \label{subsec:run theorem}

In this subsection we use Zhuang's generalized \emph{run theorem} \cite{zhu16} to get two generating functions. Here ``run'' refers to a maximal weakly increasing consecutive subsequence in a word, so naturally this theorem suits the increasing pattern $\underline{123}$ the best. But according to Proposition~\ref{omega--omega^c}, finding the Eulerian distribution over $\underline{123}$-avoiding permutations is equivalent to finding the Eulerian distribution over $\underline{321}$-avoiding permutations. And the bivariate generating function for $\underline{321}$-avoiding permutations is already known to be \cite{MR06}

\begin{align}\label{Sn321des}
	\sum_{n\geq 0}\dfrac{x^n}{n!}\sum_{\pi \in S_n(\underline{321})} t^{\mathrm{des}(\pi)}=\dfrac{\mathrm{e}^{x/2}\sqrt{1-4t}}{\sqrt{1-4t}\cosh {\left(\frac{x}{2}\sqrt{1-4t}\right)}-\sinh{\left(\frac{x}{2}\sqrt{1-4t}\right)}}.
\end{align}
See also \cite[Thm.~5.3.4]{kit11}. Our proof of the following result could thus be viewed as an alternative approach to deriving \eqref{Sn321des}. We only sketch a proof here using the run theorem and illustrate the associated ``run network'' in Fig.~\ref{sn123runnetwork}. To make the paper self-contained however, we provide some preliminary definitions and facts in the Appendix, and the reader is referred to Zhuang's original paper~\cite{zhu16} for further information.

Note that a permuation $\pi$ avoids the consecutive pattern $\underline{123}$ if and only if each run of $\pi$ has length 1 or 2. Therefore, the corresponding run network is given by Fig.~\ref{sn123runnetwork}.


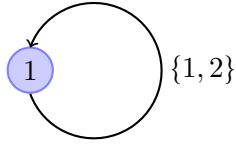
\begin{figure}[htbp]	
	 \centering
		\begin{tikzpicture}
			[phase/.style ={circle,draw=blue!50,fill=blue!20,thick,inner sep=0pt, minimum size=6mm}]
			\node (A) at (0,-2) [phase] {1};
			\node (C) at (2.3,-2) {$\{1,2\}$};
			\draw[->,thick] (A.south) arc[start angle=-160, end angle=160, radius=9mm];
		\end{tikzpicture}
\caption{Run network for permutations avoiding consecutive pattern $\underline{123}$.}
\label{sn123runnetwork}
\end{figure}

\begin{theorem}
\begin{align}\label{Sn(123)des+1}
1+\sum_{n\geq 1}\dfrac{x^n}{n!}\sum_{\pi \in \S_n(\underline{123})} t^{\mathrm{des}(\pi)+1}=\dfrac{\mathrm{e}^{tx/2}\sqrt{t-4}}{\sqrt{t-4}\cosh {\left(\frac{x}{2}\sqrt{t^2-4t}\right)}-\sqrt{t}\sinh{\left(\frac{x}{2}\sqrt{t^2-4t}\right)}},
\end{align}
equivalently, we have
\begin{align}\label{Sn(123)des}
	\sum_{n\geq 0}\dfrac{x^n}{n!}\sum_{\pi \in \S_n(\underline{123})} t^{\mathrm{des}(\pi)}=\dfrac{t^{-1}\mathrm{e}^{tx/2}\sqrt{t-4}}{\sqrt{t-4}\cosh {\left(\frac{x}{2}\sqrt{t^2-4t}\right)}-\sqrt{t}\sinh{\left(\frac{x}{2}\sqrt{t^2-4t}\right)}}-\dfrac{1}{t}+1.
\end{align}
\end{theorem}

\begin{proof}
According to Fig.~\ref{sn123runnetwork} and applying Theorem~\ref{thm:gen run thm}, the inverse of $1+xt+x^2t$ is enumerator of those permutations whose increasing runs are no more than 2, and each increasing run is weighted by $t$. Next we proceed to seek the exponential generating function expression of $$\frac{1}{1+xt+x^2t}$$ with respect to the variable $x$. Since 
$$\dfrac{1}{1+xt+xt^2}=\dfrac{t-4+\sqrt{t^2-4t}}{2\left(t-4\right)}\dfrac{1}{1-\frac{2tx}{-t+\sqrt{t^2-4t}}}+\dfrac{t-4-\sqrt{t^2-4t}}{2(t-4)}\dfrac{1}{1-\frac{2tx}{-t-\sqrt{t^2-4t}}},$$ 
the exponential generating function turns out to be 
$$\frac{t-4+\sqrt{t^2-4t}}{2\left(t-4\right)} \mathrm{Exp}\left(\frac{-tx-x\sqrt{t^2-4t}}{2}\right)+\frac{t-4-\sqrt{t^2-4t}}{2\left(t-4\right)} \mathrm{Exp}\left(\frac{-tx+x\sqrt{t^2-4t}}{2}\right)$$
whose inverse simplifies to the desired expression \eqref{Sn(123)des+1}.
\end{proof}



\subsection{The generating functions involving partial differential equations} \label{subsec:g.f.132}
Noting the equivalences implied by Proposition~\ref{omega--omega^rc}: $\underline{312} \overset{\des}{\sim }\underline{231} $, $\underline{213} \overset{\des}{\sim }\underline{132}$, and the relation $\underline{231}=(\underline{132})^r$, it suffices to compute the generating function for one of the patterns, say $\underline{132}$.

Let $U_n(t)$ (resp., $V_n(t)$) denote the generating function of the permutations $\pi$ of length $n$ that avoid the pattern $\underline{132}$ and start with an ascent $\pi_1 < \pi_2$ (resp., a descent $\pi_1 > \pi_2$), where $t$ tracks the descent number of $\pi$.
We define $U = U(x,t) := \sum_{n \geq 0} U_n(t) \frac{x^n}{n!}$ and $V = V(x,t) := \sum_{n \geq 0} V_n(t) \frac{x^n}{n!}$, with the initial terms given by $U_0(t) = U_1(t) = V_0(t) = V_1(t) = 0$.

\begin{proposition} \label{132}
The bivariate generating functions $U$ and $V$ satisfy the partial differential equations:
    \begin{align*}
    \dfrac{\partial U}{\partial x} &=(tU+1)(U+x+1)-1,\\
    \dfrac{\partial V}{\partial x} &=t(V+x)(U+x+1).
    \end{align*}
    Moreover,
    \begin{align*}
    A^{\underline{132}}(x,t)=A^{\underline{213}}(x,t)=1+x+U(x,t)+V(x,t).
    \end{align*}
\end{proposition}

\begin{proof}
    Let $\pi \in \S_{n+1}(\underline{132})$. Considering the position of $1$ as illustrated in Fig.~\ref{fig:position of 1}, we obtain
    \begin{equation}\label{eqn-prop-4-6}U_{n+1}(t)=tU_n(t)+tnU_{n-1}(t)+t\sum_{i=2}^{n-2}\binom{n}{i}U_i(t)U_{n-i}(t)+U_{n}(t),\quad n \ge 3.\end{equation}

    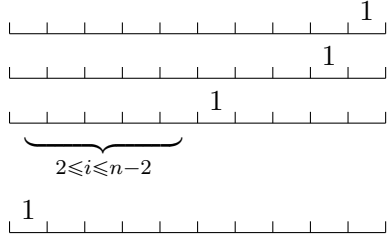
\begin{figure}[ht]
        \centering
        \begin{tikzpicture}
            \draw (0,0)--(5,0);
            \foreach \I in {0,0.5,...,5}
            {\draw (\I,0)--(\I,0.15);}
            \node at (4.75,0.3) {$1$};
        \end{tikzpicture}

        \begin{tikzpicture}
            \draw (0,0)--(5,0);
            \foreach \I in {0,0.5,...,5}
            {\draw (\I,0)--(\I,0.15);}
            \node at (4.25,0.3) {$1$};
        \end{tikzpicture}

        \begin{tikzpicture}
            \draw (0,0)--(5,0);
            \foreach \I in {0,0.5,...,5}
            {\draw (\I,0)--(\I,0.15);}
            \node at (2.75,0.3) {$1$};
            \node at (1.25,-0.3) {$\underbrace{\phantom{a+b+c+d}}_{2 \le i\le n-2}$ };
        \end{tikzpicture}

        \begin{tikzpicture}
            \draw (0,0)--(5,0);
            \foreach \I in {0,0.5,...,5}
            {\draw (\I,0)--(\I,0.15);}
            \node at (0.25,0.3) {$1$};
        \end{tikzpicture}
        \caption{Position of $1$ in $\pi \in \S_{n+1}(\underline{132})$}\label{fig:position of 1}
    \end{figure}
Multiplying by $\dfrac{x^n}{n!}$ both sides of (\ref{eqn-prop-4-6}) and summing over all $n\geq 3$, we have    

\begin{align*}
    \sum_{n\ge 3}U_{n+1}(t)\frac{x^n}{n!}={}&t\sum_{n\ge 3}U_n(t)\frac{x^n}{n!}+tx\sum_{n\ge 3}U_{n-1}(t)\frac{x^{n-1}}{(n-1)!} \\
    &+t\sum_{n\ge 3}\sum_{i=2}^{n-2}\binom{n}{i}U_i(t)U_{n-i}(t)\frac{x^n}{n!}+\sum_{n\ge 3}U_n(t)\frac{x^n}{n!} \\
    ={}& t\sum_{n\ge 3}U_n(t)\frac{x^n}{n!}+tx\sum_{n\ge 2}U_{n}(t)\frac{x^{n}}{n!} \\
    &+t\sum_{n\ge 3}\sum_{i=0}^{n}\binom{n}{i}U_i(t)U_{n-i}(t)\frac{x^n}{n!}+\sum_{n\ge 3}U_n(t)\frac{x^n}{n!} \\
    ={}& t\left(U-\frac{x^2}{2}\right)+txU+tU^2+\left(U-\frac{x^2}{2}\right).
\end{align*}

Note that $U_2(t)=1$ and $U_3(t)=1+t$. Hence,  
\begin{align*}
    \sum_{n\ge 3}U_{n+1}(t)\frac{x^n}{n!}&=\frac{\partial }{\partial x}\left(U-\frac{x^2}{2}-(1+t)\frac{x^3}{3!}\right) = \frac{\partial U}{\partial x}-x-(1+t)\frac{x^2}{2} \\
    &=t\left(U-\frac{x^2}{2}\right)+txU+tU^2+\left(U-\frac{x^2}{2}\right).
\end{align*}

Therefore, $\dfrac{\partial U}{\partial x}=x+tU+txU+tU^2+U=(tU+1)(U+x+1)-1$.
Using Mathematica, we obtain the expansion $U(x,t)=\dfrac{x^2}{2}+(1+t)\dfrac{x^3}{3!}+(1+5t+t^2)\dfrac{x^4}{4!}+(1+16t+10t^2+t^3)\dfrac{x^5}{5!}+(1+42t+71t^2+16t^3+t^4)\dfrac{x^6}{6!}+(1+99t+399t^2+197t^3+23t^4+t^5)\dfrac{x^7}{7!}+\cdots$.

Similarly, we have
$$V_{n+1}(t)=tV_n(t)+tnV_{n-1}(t)+t\sum_{i=2}^{n-2}\binom{n}{i}V_i(t)U_{n-i}(t)+tnU_{n-1}(t),\quad n\ge 3.$$
Note that $V_2(t)=t$ and $V_3(t)=2t+t^2$. Repeating the above computations for $V$, we have

\begin{align*}
    \sum_{n\ge 3}V_{n+1}(t)\frac{x^n}{n!}={}& t\sum_{n\ge 3}V_n(t)\frac{x^n}{n!}+tx\sum_{n\ge 3}V_{n-1}(t)\frac{x^{n-1}}{(n-1)!} \\
    &+t\sum_{n\ge 3}\sum_{i=2}^{n-2}\binom{n}{i}V_i(t)U_{n-i}(t)\frac{x^n}{n!}+tx\sum_{n\ge 3}U_{n-1}(t)\frac{x^{n-1}}{(n-1)!} \\
    ={}& t\sum_{n\ge 3}V_n(t)\frac{x^n}{n!}+tx\sum_{n\ge 2}V_{n}(t)\frac{x^{n}}{n!} \\
    &+t\sum_{n\ge 3}\sum_{i=0}^{n}\binom{n}{i}V_i(t)U_{n-i}(t)\frac{x^n}{n!}+tx\sum_{n\ge 2}U_{n}(t)\frac{x^{n}}{n!} \\
    ={} & t\left(V-\frac{tx^2}{2}\right)+txV+tUV+txU.
\end{align*}
Then,
\begin{align*}
    \sum_{n\ge 3}V_{n+1}(t)\frac{x^n}{n!} &= \frac{\partial }{\partial x}\left(V-t\frac{x^2}{2}-(2t+t^2)\frac{x^3}{3!}\right)=\frac{\partial V}{\partial x} -tx-tx^2\left(1+\frac{t}{2}\right) \\
    &=t\left(V-\frac{tx^2}{2}\right)+txV+tUV+txU.
\end{align*}

Therefore, $\dfrac{\partial V}{\partial x}=tx+tx^2+tV+txV+tVU+txU=t(V+x)(U+x+1)$.
Using Mathematica, $V(x,t)=t\dfrac{ x^2}{2}+(2t+t^2)\dfrac{x^3}{3!}+(3t+5t^2+t^3)\dfrac{x^4}{4!}+(4t+21t^2+9t^3+t^4)\dfrac{x^5}{5!}+\cdots$.
\end{proof}

\begin{corollary}\label{231}
We have
    \begin{align*}
    A^{\underline{231}}(x,t)=A^{\underline{312}}(x,t)=\dfrac{1}{t} \left(U(xt,\dfrac{1}{t})+V(xt,\dfrac{1}{t})\right)+x+1.
    \end{align*}
\end{corollary}

\begin{proof}
    Applying Proposition~\ref{omega--omega^c}, we obtain
    \begin{equation*}
    A^{\underline{231}}(x,t)=\dfrac{1}{t} \left(A^{\underline{132}}(xt,\dfrac{1}{t})-1\right)+1=\dfrac{1}{t} \left(xt+U(xt,\dfrac{1}{t})+V(xt,\dfrac{1}{t})\right)+1. \qedhere
    \end{equation*}
\end{proof}

\section{Conclusion}\label{conclusion}

This paper provides a complete $\des$-Wilf classification for quasi-consecutive patterns of length up to 4. Extending this classification to longer patterns would be interesting but may not be feasible with existing tools. Indeed, determining the descent distributions for consecutive patterns of length 4 or more is already a complex task, and even if successful, evaluating the integral in Theorem~\ref{id:struc} may not be possible. Nevertheless, descent distributions can still be studied over more restrictive sets of permutations, specifically when, in addition to avoiding a quasi-consecutive pattern, further constraints are imposed (e.g., avoidance of another pattern, a prescribed beginning or end of a permutation, etc.).

Finally, studying other permutation statistics, such as left-to-right and right-to-left minima and maxima  (see \cite{HK} for definitions), on permutations that avoid quasi-consecutive patterns would be an interesting direction for future research, extending our study of descents on such permutations.

\section{appendix}\label{sec:appendix}

To state the generalized run theorem, we start by introducing some definitions on words and compositions, then we review some basic facts about noncommutative symmetric functions. 

Given an alphabet $A$, the set of all words over $A$ is denoted by $A^*$, which is a free monoid under the operation of concatenation. If $A$ is a totally ordered set and $w \in A^*$, we call maximal weakly increasing consecutive subsequence in $w$ an \emph{increasing run} of $w$. It is clear that each word $w$ can be decomposed uniquely into a sequence of increasing runs. For example, let $\mathbf{P}$ denote the set of positve integers, then the word $w=73344256 \in \mathbf{P}^*$ is decomposed into three increasing runs, namely $7$, $3344$, and $256$.

Reall that for a permutation $\pi \in \S_n$, $\mathrm{Des}(\pi):=\{i \in [n-1]:\pi_i>\pi_{i+1}\}$ denotes the descent set of $\pi$. Clearly $\mathrm{Des}(\pi) \subseteq [n-1]$ and the lengths of increasing runs in $\pi$ from left to right constitute a composition of $n$. We call this composition the \emph{descent composition} of $\pi$, denoted as $\mathrm{Comp}(\pi)$. For example if $\pi=413782596$, then $\mathrm{Des}(\pi)=\{1,5,8\}$ and $\mathrm{Comp}(\pi)=\{1,4,3,1\}$. Obviously one can recover $\mathrm{Des(\pi)}$ from $\mathrm{Comp}(\pi)$ or the other way around. For $n\geq 1$ and $\pi \in \S_n$, we have $$|\mathrm{Comp}(\pi)|=\mathrm{des}(\pi)+1=|\mathrm{Des(\pi)}|+1.$$

Zhuang's generalized run theorem was developed to count words and permutations with restrictions on the lengths of their increasing runs.

Let $F$ be a field of characteristic zero. Denote by $F\left\langle\left\langle X_1, X_2, \ldots\right\rangle\right\rangle$ the $F$-algebra of formal power series in noncommuting variables $X_1, X_2, \ldots$. For $n \geq 1$, let

$$
\mathbf{h}_n:=\sum_{1 \leq i_1 \leq \cdots \leq i_n} X_{i_1} X_{i_2} \cdots X_{i_n}
$$and set $\mathbf{h}_0=1$. Then for any composition $L=\left(L_1, \ldots, L_l\right)$, let $\mathbf{h}_L:=\mathbf{h}_{L_1} \cdots \mathbf{h}_{L_l}$. The $F$-algebra $\mathbf{S y m}$ generated by the elements $\mathbf{h}_L$ is called the algebra of noncommutative symmetric functions with coefficients in $F$, which is a subalgebra of $F\left\langle\left\langle X_1, X_2, \ldots\right\rangle\right\rangle$. The vector space $\mathbf{S y m}_n$ of noncommutative symmetric functions homogeneous of degree $n$ is the span of the set $\left\{\mathbf{h}_L\right\}_{L \vdash n}$, where $L\vdash n$ means $L$ is a composition of $n$.

Next, we are going to introduce another base $\left\{\mathbf{r}_L\right\}_{L \vdash n}$ for the vector space of noncommutative symmetric functions homogeneous of degree $n$, see \cite{GKLLRT95} for more information about noncommutative symmetric functions. 

For any composition $L=\left(L_1, \ldots, L_l\right) \vdash n$, let

$$
\mathbf{r}_L:=\sum_{\left(i_1, \ldots, i_n\right)} X_{i_1} X_{i_2} \cdots X_{i_n}
$$where the sum is over all $\left(i_1, \ldots, i_n\right)$ satisfying

$$
\underbrace{i_1 \leq \cdots \leq i_{L_1}}_{L_1}>\underbrace{i_{L_1+1} \leq \cdots \leq i_{L_1+L_2}}_{L_2}>\cdots>\underbrace{i_{L_1+\cdots+L_{l-1}+1} \leq \cdots \leq i_n}_{L_l} .
$$

Define a partial order ``$\leq$'' on the set of compositions of $n$ by reverse refinement, that is, $K=\left(K_1, \ldots, K_k\right) \vdash n$ covers $L$ if and only if $L$ can be obtained from $K$ by replacing two consecutive parts $K_i$ and $K_{i+1}$ with $K_i+K_{i+1}$. It is easy to see that

$$
\mathbf{h}_L=\sum_{K \leq L} \mathbf{r}_K.
$$Thus, by inclusion-exclusion we have

$$
\mathbf{r}_L=\sum_{K \leq L}(-1)^{\#L-\#K} \mathbf{h}_K,
$$where $\#L$ denotes the number of parts in $L$. Note that the functions $\mathbf{h}_L$ and $\mathbf{r}_L$ are noncommutative versions of the \emph{complete symmetric functions} and the \emph{ribbon Schur functions} (see \cite[Chap.~7]{sta11}), respectively.

Define the homomorphism $\Phi: \mathbf{S y m} \rightarrow F[[x]]$ by $\Phi\left(\mathbf{h}_n\right)=\dfrac{x^n}{n!}$. Then, $$\Phi(\mathbf{h}_{L})=\frac{x^{L_1}}{L_1 !}\cdots \frac{x^{L_l}}{L_l !}=\binom{n}{L}\frac{x^n}{n!},$$
where we use the abbreviation for multinomial coefficient $\binom{n}{L}:=\binom{n}{L_1,\ldots,L_l}$. Although noncommutative symmetric functions are generating functions for words, the following lemma allows us to move from the realm of word enumeration to that of permutation enumeration.

\begin{lemma}{\cite{zhu17}}
For a given composition $L\vdash n$, we have  
\begin{equation}\label{homomorphsim}
\Phi (\mathbf{r}_L)=\sum_{\substack{\pi \in \S_n\\ \mathrm{Comp}(\pi)=L}} \frac{x^n}{n!}
\end{equation}
\end{lemma}

\begin{Def}{(\emph{run network})} A digraph $G$ on vertex set $[m]$ where each arc $(i, j)$ is assigned a nonempty subset $P_{i, j} \subseteq \mathbf{P}$ is called a network of order $m$. Given such a network, let $P$ be the set of all pairs $(a, e)$ such that $a \in P_{i, j}$ and $e=(i, j)$ is an arc of $G$. Furthermore, let $\overrightarrow{P^*} \subseteq P^*$ be the subset of all sequences $\alpha=\left(a_1, e_1\right)\left(a_2, e_2\right) \cdots\left(a_n, e_n\right)$ such that $e_1 e_2 \cdots e_n$ is a walk in $G$.

Let $(G, P)$ be a network of order $m$ with $\overrightarrow{P^*}$ as constructed above. For $\alpha=\left(a_1, e_1\right) \cdots\left(a_n, e_n\right) \in \overrightarrow{P^*}$ with $i$ and $j$ respectively the initial and terminal vertices of the walk $e_1 \cdots e_n$, let $\rho(\alpha)=\left(a_1, a_2, \ldots, a_n\right)$ and $E(\alpha)=(i, j)$. Then $(G, P)$ is a run network of order $m$ if for any $\alpha, \beta \in \overrightarrow{P^*}$, whenever $\rho(\alpha)=\rho(\beta)$ and $E(\alpha)=E(\beta)$, there is $\alpha=\beta$. In other words, the same tuple $\left(a_1, a_2, \ldots, a_n\right)$ cannot be obtained by traversing two different walks in $G$ with the same initial and terminal vertices.
\end{Def}

\begin{theorem}{(Generalized Run Theorem~\cite{zhu17})}\label{thm:gen run thm}
  Suppose that $(G, P)$ is a run network of order $m$ and $A$ is a unital $F$-algebra of characteristic zero. Let $\left\{w_a^{(i , j)}:(a,(i, j)) \in P\right\}$ be a set of weights from $A$, with $w_a^{(i, j)}=0$ if $(a,(i, j)) \notin P$. For a composition $L=\left(L_1, \ldots, L_l\right)$ and $i, j \in[m]$, let $w^{(i, j)}(L)=w_{L_1}^{e_1} \cdots w_{L_l}^{e_l}$ if there exists $\alpha=\left(L_1, e_1\right) \cdots\left(L_l, e_l\right) \in \overrightarrow{P^*}$ such that $E(\alpha)=(i, j)$ and $L=\rho(\alpha)$, and let $w^{(i, j)}(L)=0$ otherwise. Let $v_n^{(i, j)}$ be defined by
    
\begin{align*}
      \left[\begin{array}{ccc}
      \sum\limits_{n \geq 0} v_n^{(1,1)} x^n & \cdots & \sum\limits_{n \geq 0} v_n^{(1, m)} x^n \\
      \vdots & \ddots & \vdots \\
      \sum\limits_{n \geq 0} v_n^{(m, 1)} x^n & \cdots & \sum\limits_{n \geq 0} v_n^{(m, m)} x^n
  \end{array}\right]=\left(I_m+\left[\begin{array}{ccc}
      \sum\limits_{k \geq 1} w_k^{(1,1)} x^k & \cdots & \sum\limits_{k \geq 1} w_k^{(1, m)} x^k \\
      \vdots & \ddots & \vdots \\
      \sum\limits_{k \geq 1} w_k^{(m, 1)} x^k & \cdots & \sum\limits_{k \geq 1} w_k^{(m, m)} x^k
  \end{array}\right]\right)^{-1},
\end{align*}

  where $I_m$ denotes the $m \times m$ identity matrix. Then
    
  \begin{align}
      \left[\begin{array}{ccc}\label{nonComGen}
      \sum\limits_L w^{(1,1)}(L) \mathbf{r}_L & \cdots & \sum\limits_L w^{(1, m)}(L) \mathbf{r}_L \\
      \vdots & \ddots & \vdots \\
      \sum\limits_L w^{(m, 1)}(L) \mathbf{r}_L & \cdots & \sum\limits_L w^{(m, m)}(L) \mathbf{r}_L
  \end{array}\right]=\left[\begin{array}{ccc}
      \sum\limits_{n \geq 0} v_n^{(1,1)} \mathbf{h}_n & \cdots & \sum\limits_{n \geq 0} v_n^{(1, m)} \mathbf{h}_n \\
      \vdots & \ddots & \vdots \\
      \sum\limits_{n \geq 0} v_n^{(m, 1)} \mathbf{h}_n & \cdots & \sum\limits_{n \geq 0} v_n^{(m, m)} \mathbf{h}_n
  \end{array}\right]^{-1},
  \end{align}

  where each sum in the matrix on the left-hand side is over all compositions $L$.
\end{theorem}

\section*{acknowledgement}
We are grateful to Quan Yuan for her help in finding the $\des$-Wilf-equivalence classifications
for quasi-consecutive patterns of length $4$.

\end{document}